\documentclass[11pt,a4paper]{amsart}
\usepackage[T1]{fontenc}
\usepackage{microtype}
\usepackage{lmodern}
\usepackage[colorlinks=true,urlcolor=blue, citecolor=red,linkcolor=blue,linktocpage,pdfpagelabels, bookmarksnumbered,bookmarksopen]{hyperref}
\usepackage[hyperpageref]{backref}
\usepackage{amsthm} 
\usepackage{latexsym,amsmath,amssymb}

\usepackage{accents}
\usepackage{esint}

\usepackage{soul}
\usepackage{mathtools} 
\usepackage{xparse} 
\usepackage[capitalize]{cleveref}

\title{An estimate of the Hopf degree of fractional Sobolev mappings}
\author{Armin Schikorra}
\address[Armin Schikorra]{Department of Mathematics,
University of Pittsburgh,
301 Thackeray Hall,
Pittsburgh, PA 15260, USA}
\email{armin@pitt.edu}

\thanks{A.S. gratefully acknowledges support by the Simons foundation, grant no 579261.}
\author{Jean Van Schaftingen}
\address[Jean Van Schaftingen]{Universit\'e catholique de Louvain (UCLouvain)\\ 
  Institut de Recherche en Math\'ematique et Physique (IRMP)\\
  Chemin du Cyclotron 2 bte L7.01.01\\
  1348 Louvain-la-Neuve\\
  Belgium}
\email{Jean.VanSchaftingen@uclouvain.be}
\thanks{J.V.S. gratefully acknowledges support by Fonds de la Recherche Scientifique--FNRS, Mandat d'Impulsion Scientifique F.4523.17, ``Topological singularities of Sobolev maps''.}
\newcommand{\B}{{\mathbb{B}}}
\newcommand{\Bset}{{\mathbb{B}}}
\newcommand{\N}{{\mathbb N}}
\newcommand{\Nset}{\mathbb{N}}
\renewcommand{\S}{{\mathbb S}}
\newcommand{\Sset}{{\mathbb{S}}}

\newtheorem{theorem}{Theorem}
\newtheorem{lemma}[theorem]{Lemma}
\newtheorem{corollary}[theorem]{Corollary}
\newtheorem{proposition}[theorem]{Proposition}

\theoremstyle{definition}

\theoremstyle{remark}
\newtheorem{remark}[theorem]{Remark}

\newcommand\supp{{\rm supp\,}}

\newcommand{\R}{\mathbb{R}}
\newcommand{\Rset}{\mathbb{R}}

\newcommand{\Z}{\mathbb{Z}}

\newcommand{\brac}[1]{\left (#1 \right )}
\newcommand{\abs}[1]{\left\lvert #1 \right \rvert}
\newcommand{\biggabs}[1]{\biggl\lvert #1 \biggr\rvert}
\newcommand{\Ep}{\bigwedge\nolimits}
\newcommand{\norm}[1]{\left\|{#1}\right\|}

\newcommand{\dif}{\,\mathrm{d}}

\newcommand{\compose}{\,\circ\,}

\newcommand{\barint}{
\rule[.036in]{.12in}{.009in}\kern-.16in \displaystyle\int }

\newcommand{\barcal}{\text{$ \rule[.036in]{.11in}{.007in}\kern-.128in\int $}}

\def\mvint_#1{\mathchoice
          {\mathop{\vrule width 6pt height 3 pt depth -2.5pt
                  \kern -8pt \intop}\nolimits_{\kern -3pt #1}}%
          {\mathop{\vrule width 5pt height 3 pt depth -2.6pt
                  \kern -6pt \intop}\nolimits_{#1}}%
          {\mathop{\vrule width 5pt height 3 pt depth -2.6pt
                  \kern -6pt \intop}\nolimits_{#1}}%
          {\mathop{\vrule width 5pt height 3 pt depth -2.6pt
                  \kern -6pt \intop}\nolimits_{#1}}}

\numberwithin{theorem}{section} \numberwithin{equation}{section}

\newcommand{\lap}{\Delta }
\newcommand{\aleq}{\lesssim}

\newcommand{\aeq}{\approx}
\newcommand{\defeq}{:=}

\newcommand{\lapms}[1]{\mathcal{I}_{#1}}

\let\latexchi\chi
\makeatletter
\renewcommand\chi{\@ifnextchar_\sub@chi\latexchi}
\newcommand{\sub@chi}[2]{
  \@ifnextchar^{\subsup@chi{#2}}{\latexchi^{}_{#2}}%
}
\newcommand{\subsup@chi}[3]{
  \latexchi_{#1}^{#3}%
}
\makeatother

\begin{document}

\begin{abstract}
We estimate the Hopf degree for smooth maps $f$ from $\S^{4n-1}$ to $\S^{2n}$ in the fractional Sobolev space.
Namely we show that for $s \in [1 - \frac{1}{4n}, 1]$
\begin{equation*}
  \left |\deg_H(f)\right | \aleq [f]_{W^{s,\frac{4n-1}{s}}}^{\frac{4n}{s}}.
\end{equation*}
Our argument is based on the Whitehead integral formula and commutator estimates for Jacobian-type expressions.
\end{abstract}

\maketitle

\sloppy

\section{Introduction}

\subsection*{Brouwer degree estimates}
For smooth maps $f: \S^n \to \S^n$ it is well-known that the Brouwer topological degree can be computed by the formula
\begin{equation}
\label{eq_voi0ooZaequi3iet}
\deg f = \int_{\S^n} f^\ast\omega_{\Sset^n},
\end{equation}
where \(f^\ast (\omega_{\Sset^n})\) is the pull-back of the volume form $\omega_{\Sset^n}$ on the sphere $\S^n$, which by an extension argument can be interpreted as a restriction of an $n$-form $\omega \in C_c^\infty(\Ep^n T^\ast \R^{n+1})$.

The estimate \eqref{eq_voi0ooZaequi3iet} readily implies that the degree can be estimated by the norm in the critical Sobolev space \(W^{1, n} (\Sset^n, \Sset^n)\) since $|f^\ast\omega| \aleq \norm{\omega}_{L^\infty} \, |Df|^n$ pointwise everywhere on \(\Sset^n\) for any $n$-form $\omega \in C_c^\infty(\Ep^n T^\ast \R^{n+1})$.

In the fractional case, since $f^\ast\omega$ is of Jacobian type, we have for every \(s \ge \frac{n}{n + 1}\) and every \(f : \Sset^n \to \Rset^{n + 1}\) commutator estimates from Harmonic Analysis 
of the form (see \cref{pr:degest})
\begin{equation}
  \label{eq_noc7ahX9iapheejai}
    \biggabs{\int_{\Sset^n} f^* (\omega)} 
  \le 
    C(n, s)
    \norm{d\omega}_{L^\infty (\Rset^{n + 1})} 
    \|f\|_{L^\infty(\S^n)}^{n + 1 - \frac{n}{s}}
    \,
    [f]_{W^{s,\frac{n}{s}}(\S^n)}^{\frac{n}{s}}
  .
\end{equation}
This implies the degree estimate
\begin{equation}
\label{eq_eequ5shuinoh3Eum}
  \abs{\deg f}   \aleq  
  [f]_{W^{s,\frac{n}{s}}(\S^n)}^{\frac{n}{s}}.
\end{equation}
for any $s \geq \frac{n}{n+1}$. Here $W^{s,p}(\S^n)$ denotes the fractional Sobolev space with the semi-norm
\begin{equation*}
 [f]_{W^{s,p}(\S^n)} = \left ( \int_{\S^{n}} \int_{\S^n} \frac{|f(x)-f(y)|^p}{|x-y|^{n+sp}}\dif x \dif y\right )^{\frac{1}{p}}.
\end{equation*}

The threshold $s \geq \frac{n}{n+1}$ is sharp from the point of view of Harmonic Analysis: without the condition that $f$ maps into the sphere $\S^{n}$ one cannot estimate $\int_{\S^n} f^\ast\omega_{\Sset^n}$ in terms of the $W^{s,\frac{n}{s}}$-norm for any $s < \frac{n}{n+1}$, see \cref{pr:deg:blowup}.
When \(s < \frac{n}{n+1}\), although the estimate of \(\int_{\S^n} f^\ast\omega_{\Sset^n}\) fails, 
the estimate \eqref{eq_eequ5shuinoh3Eum} still holds, with a subtle proof based on adequate estimate of the singular set of a harmonic extension of the mapping \(f\) \cite[Theorem 0.6]{BBM05}; this strategy also yields gap potential estimates 
\cite{BBM05,BBN05,Nguyen07}.

\subsection*{Hopf degree estimates}
Hopf \cite{Hopf} showed that for $n \in \N$ maps \(f : \Sset^{4n - 1} \to \Sset^{2n}\) have a topological invariant, the \emph{Hopf degree} or \emph{Hopf invariant}. Whitehead \cite{W47} introduced an elegant integral formula for the Hopf degree:
\begin{equation}\label{eq:deghdef}
  \deg_H f := \int_{\S^{4n - 1}} \eta \wedge f^\ast (\omega_{\Sset^{2n}}),
\end{equation}
where $d\eta = f^\ast\omega_{\Sset^{2n}} \in C^\infty (\Ep^{2n - 1} T^* \Sset^{4n - 1})$. 
The form $\eta$ exists by Poincar\'e's lemma, since $f^\ast d\omega = 0$ and the $2n$ de Rham cohomology group of $\S^{4n-1}$ is trivial: \(H_{\mathrm{dR}}^{2n} (\S^{4n - 1}) \simeq \{0\}\). 
The Hopf invariant does not depend on the choice of \(\eta\) and is invariant under homotopies, see e.g. \cite{BT82}.

The Hopf invariant can be estimated by the critical Sobolev norm of \(W^{1, 4n - 1} (\Sset^{4n - 1}, \Sset^{2n})\), \cite[Lemma III.1]{R98},
\begin{equation}
\label{eq_wieyaiL3ahpo8ja1}
  \abs{\deg_H f} \aleq [f]_{W^{1, 4 n - 1}(\Sset^{4n - 1}, \Sset^{2n})}^{4n}.
\end{equation}
See also \cite{Gromov_1999}, \cite[Lemma 7.12]{Gromov_1999_MS} for corresponding estimates with the Lipschitz seminorm, and \cite{HST14} for related estimates for maps with low rank.

Remarkably, the exponent in \eqref{eq_wieyaiL3ahpo8ja1} is different from the one in \eqref{eq_eequ5shuinoh3Eum}. In \cite[Theorem 1.3]{VS2020}, by a compactness argument, it has been established that bounded sets in critical Sobolev spaces \(W^{s, \frac{4n - 1}{s}} (\Sset^{4 n  - 1}, \Sset^{2n})\) are generated finitely up to the action of the fundamental group \(\pi_1 (\Sset^{2n})\) on the homotopy group \(\pi_{4 n - 1} (\Sset^{2n})\) (see also \cite[Theorem 5.1]{Mueller2000} for \(s = 1 - \frac{1}{4n}\)). Since \(\pi_1 (\Sset^{2n})\) is trivial, 
the Hopf degree is bounded on bounded sets of \(W^{s, \frac{4n - 1}{s}} (\Sset^{4 n  - 1}, \Sset^{2n})\).

Our main theorem is an extension of the estimate \eqref{eq_wieyaiL3ahpo8ja1} to fractional Sobolev spaces.
\begin{theorem}
\label{th:deghest}
Let $s \in [1 - \frac{1}{4n}, 1]$. 
Let $f: \S^{4n-1} \to \S^{2n}$ be a smooth map, then we have the estimate
\begin{equation*}
  \left |\deg_H(f)\right | \leq C(n,s)\, [f]_{W^{s,\frac{4n-1}{s}}(\S^{4n-1})}^{\frac{4n}{s}}.
\end{equation*}
\end{theorem}
The proof of \cref{th:deghest} is based on commutator estimates, i.e.\ tools from Harmonic Analysis which disregard the topological condition that $f$ maps into $\S^{2n}$. Namely \cref{th:deghest} is a consequence of the counterpart of \eqref{eq_noc7ahX9iapheejai} for the Hopf degree
\begin{theorem}\label{th:Hdeg}
Let $n\geq 1$ and $\omega \in C^1(\Ep^{2n} T^\ast \R^{2n+1})$ and $s \in [\frac{4n-1}{4n},1]$. For any $f \in C^\infty(\S^{4n-1}, \R^{2n+1})$ and any $\eta$ such that $f^\ast\omega = d\eta$ we have 
\begin{equation*}
  \biggl\lvert 
\int_{\R^{4n - 1}} 
\eta
\wedge 
f^* \omega 
\biggr\rvert 
\aleq \norm{\omega}_{L^\infty }^{2}
\|f\|_{L^\infty}^{4n-\frac{4n-1}{s}} 
 [f]_{W^{s,\frac{4n-1}{s}}}^{\frac{4n-1}{s}} 
\biggl( 1 + 
\frac{\norm{D\omega}_{L^\infty}[f]_{W^{s, \frac{4n - 1}{s}}}}{\norm{\omega}_{L^\infty }} \biggr)^\frac{1}{s}.
\end{equation*}
\end{theorem}
\begin{proof}[Deducing \cref{th:deghest} from \cref{th:Hdeg}]
Since the Hopf degree $\deg_H f$ is an integer, we obtain from \cref{th:Hdeg} that there exists some $\varepsilon_0 > 0$ such that
\(
[f]_{W^{s,\frac{4n-1}{s}}}<\varepsilon_0
\)
implies $\deg_H f = 0$, and the claim holds in that case. If on the other hand
\(
[f]_{W^{s,\frac{4n-1}{s}}}\geq \varepsilon_0
\),
then 
\begin{equation*}
\brac{[f]_{W^{s,\frac{4n-1}{s}}}^{4n} +  [f]_{W^{s,\frac{4n-1}{s}}}^{\frac{4n}{s}}} \aleq [f]_{W^{s,\frac{4n-1}{s}}}^{\frac{4n}{s}},
\end{equation*}
and thus the claim follows also in this case.
\end{proof}

The proof of \cref{th:Hdeg} is given in Section~\ref{s:nasty}. It is crucially based on commutator estimates, namely \cref{pr:Jaces}. On their own these estimates are sharp, see \cref{rem:prjaces}.

The exponent $\frac{4n}{s}$ in \cref{th:deghest} is sharp, as is shown in \cref{pr:degh:powersharp}. In \cref{pr:degh:blowup} we show that an estimate such as \cref{th:Hdeg} cannot holds for $s < \frac{2n}{2n+1}$, i.e.\ as in the case of the degree we find  an analytical threshold. However, it is below our estimate, and we did not find a counterexample for $\frac{2n}{2n+1} \leq s < \frac{4n-1}{4n}$, corresponding in particular when \(n=1\) to \(\frac{2}{3} \le s < \frac{3}{4}\).

Also, we were not able to prove \cref{th:Hdeg} via the elegant extension argument as for the degree in \cite{Brezis-Nguyen-2011}, or using a similar extension argument that works for a large class of commutator estimates \cite{LS18}. One can use such an extension argument, but it leads to a larger threshold for $s$ up to which the estimate can be shown, such an estimate was obtained for H\"older maps in \cite{HMS19}.

We conclude this section with two remarks.
\begin{remark}
It seems very likely that our argument for the proof of \cref{th:deghest} can be adapted to estimate the rational homotopy class of a smooth $f: \S^n \to \mathcal{N}$ for a manifold $\mathcal{N}$. Indeed, an integral formula similar to the Whitehead formula for the Hopf degree is known in that case, see \cite[Section~2.2]{HR08}.
\end{remark}

\begin{remark}
The difference in Harmonic Analysis threshold versus topological threshold seems also to be connected to the Gromov conjecture on embeddings $\varphi$ of the two-dimensional ball $\B^2$ into the Heisenberg group $\mathbb{H}_1$, \cite[3.1.A]{G96}. Gromov showed that such an embeddings cannot be of class $C^{\frac{2}{3}+\varepsilon}$ for any $\varepsilon > 0$, see also \cite{P16}, and conjectured that they actually cannot be of class $C^{\frac{1}{2}+\varepsilon}$ for any $\varepsilon > 0$. In \cite{HMS19} it is shown that the $C^{\frac{2}{3}+\varepsilon}$-threshold proved by Gromov actually holds for all maps that are extensions of an embedding of $\S^1 \to \mathbb{H}_1$, i.e.\ without the topological assumption that the map $\varphi$ is an embedding in $\B^2$. Moreover, the results in \cite{WengerYoung18} suggest that this $C^{\frac{2}{3}}$-threshold might be sharp without that embedding assumption. So again we are here in a situation where, without a restrictive topological assumption, the optimal class is $C^{\frac{2}{3}}$ and with topological assumption it is conjectured that $C^{\frac{1}{2}}$ is the optimal class. See also the survey \cite{S18}.
\end{remark}

\section{Degree Estimates and Jacobians}

In order to explain our proof of the Hopf degree estimate, \cref{th:Hdeg}, we first explain a strategy to prove the degree estimate \eqref{eq_eequ5shuinoh3Eum} when \(s \geq \frac{n}{n + 1}\). 
\begin{proposition}\label{pr:degest}
If $s \geq \frac{n}{n+1}$, then for every $f \in C^\infty(\S^n,\R^\ell)$ and $\omega \in C^\infty(\Ep^n T^\ast \R^\ell)$,
\begin{equation*}
    \left | \int_{\S^n} f^\ast\omega \right | 
  \leq 
    C(n,s)
    \, 
    \norm{d\omega}_{L^\infty (\Rset^\ell)} 
    \,
    \|f\|_{L^\infty(\S^n)}^{n + 1 - \frac{n}{s}}
    \,
    [f]_{W^{s,\frac{n}{s}}(\S^n)}^{\frac{n}{s}}
    .
\end{equation*}
\end{proposition}
\begin{proof}
Our proof is based on trace estimates for harmonic extensions, see \cite{BN11}, \cite[(2.3)]{Nguyen_2014}, \cite[Section~10]{LS18}, \cite{Gladbach_Olbermann}.
Alternatively, one could resort to heavier Harmonic Analysis (namely Littlewood--Paley projections and paraproducts), as in \cite{SY99}, or to Fourier Analysis, as in \cite{T85}, to obtain the same estimate.

Let $F: \B^{n+1} \to \R^\ell$ be the harmonic extension of $f: \S^n \to \R^\ell$. Then, by Stokes theorem,
\begin{equation*}
 \left | \int_{\S^n} f^\ast\omega \right | = \left | \int_{\B^{n+1}} f^\ast d\omega \right | \aleq 
 \norm{d\omega}_{L^\infty}
 \int_{\B^{n+1}} |DF|^{n+1}.
\end{equation*}
Since $F$ is the harmonic extension of \(f\), we have the trace estimate
\begin{equation*}
 \|DF\|_{L^{n+1}(\B^{n+1})}^{n+1} \aleq [f]_{W^{\frac{n}{n+1},n+1}(\S^n)}^{n+1}.
\end{equation*}
Consequently, for any $s > \frac{n}{n+1}$, by the fractional Gagliardo--Nirenberg interpolation inequality, \cite{BM17},
\begin{equation*}
 \|DF\|_{L^{n+1}(\B^{n+1})}^{n+1} \aleq \|f\|_{L^\infty(\S^n)}^{n+1-\frac{n}{s}}\, [f]_{W^{s,\frac{n}{s}}(\S^n)}^{\frac{n}{s}}.
\end{equation*}
This shows
\begin{equation*}
\left | \int_{\S^n} f^\ast\omega \right | \leq C(n,s)\,\norm{d\omega}_{L^\infty (\Rset^\ell)} \|f\|_{L^\infty(\S^n)}^{n+1-\frac{n}{s}}\, [f]_{W^{s,\frac{n}{s}}(\S^n)}^{\frac{n}{s}}.
\end{equation*}
and the claim is proven.
\end{proof}
Taking $\omega$ as the volume form of $\S^n$ in \cref{pr:degest} we obtain in particular the estimate
\begin{corollary}\label{co:degest}
If \(s \geq \frac{n}{n+1}\), then for all $f \in C^\infty(\S^n,\S^n)$, 
\begin{equation}\label{eq:degest}
 |\deg f| \aleq  C(n,s)\, [f]_{W^{s,\frac{n}{s}}(\S^n)}^{\frac{n}{s}}.
\end{equation}
\end{corollary}
The power in \cref{co:degest} is sharp in the following sense.
\begin{proposition}\label{pr:deg:powersharp}
  Let \(s \in (0, 1]\).
For any $d \in \Z$ there exists a map $f \in C^\infty(\S^n,\S^n)$ such that 
\begin{align*}
 \deg f_d &= d&
 &\text{ and}&
 [f_d]_{W^{s,\frac{n}{s}}(\S^n)}^{\frac{n}{s}} &\aleq |d|.
\end{align*}
\end{proposition}
\Cref{pr:deg:powersharp} is a consequence of the following Lemma.
\begin{lemma}
  \label{lem:deg:powersharp}
  Let \(s \in (0, 1]\).
  For any $d \in \Z$ there exists a map $f \in C^\infty(\S^n,\S^n)$ such that 
  \begin{align*}
    \deg f_d &= d&
    &\text{ and }&
    \norm{Df}_{L^\infty} \aleq |d|^{1/n}.
  \end{align*}
\end{lemma}

\begin{proof}
For every \(d \in \N\), there exists a map 
\(f_d \in C^\infty (\Sset^{n}, \Sset^{n})\) such that \(\deg f_d = d\)
and 
\(\abs{D f_d} \aleq d^{-1/n}\) on \(\Sset^{n}\).
Indeed, the sphere \(\Sset^{n}\) contains \(d\) disjoint geodesic balls \((B_{\rho_d} (a_i))_{1 \le i \le d}\) of radius \(\rho_d \aleq d^{1/{n}}\); 
it is possible to define for each \(i \in \{1, \dotsc, d\}\) a map 
\(v_d^i : \Sset^n \to \Sset^n\) such that \(v_d^i = b\) in \(\Sset^{n} \setminus B_{\rho_j} (a_i)\),
\(\deg v_d = 1\) and \(\abs{D v_d^i} \aleq 1/\rho_d\); we define then \(f_d \defeq v_d^i\) in \(B_{\rho_j} (a_j^i)\) and \(v_d = b\) otherwise. See, e.g., \cite[Lemma III.1]{R98}. 
\end{proof}
\begin{proof}[Proof of \cref{pr:deg:powersharp}]
By \cref{lem:deg:powersharp}, we have 
\begin{equation*}
 \|\nabla f_d\|_{L^n(\S^n)}^n \aleq d.
\end{equation*}
By the fractional Gagliardo--Nirenberg interpolation inequality \cite{BM17}, we have
\begin{equation*}
[f_d]_{W^{s,\frac{n}{s}}(\S^n)}^{\frac{n}{s}} \aleq d.\qedhere
\end{equation*}
\end{proof}
The differentiability threshold $s \geq \frac{n}{n+1}$ in \cref{pr:degest} is sharp from the point of view of the Harmonic Analysis involved: without the assumption that $f$ maps into $\S^n$ there is no way to lower the differential order $s$ below $\frac{n}{n+1}$.

\begin{proposition}
  \label{pr:deg:blowup}
Let $\omega$ be the volume form of $\S^n$ and let $s \in (0,\frac{n}{n+1})$. Then there exists a sequence $f_k \in C^\infty(\S^n,\R^{n+1})$ such that
\begin{equation*}
 \sup_{k \in \N} \|f_k\|_{L^\infty(\S^n)} + [f_k]_{W^{s,\frac{n}{s}}(\S^n)} < +\infty
\end{equation*}
but
\begin{equation*}
 \int_{\S^n} f_k^\ast\omega \xrightarrow{k \to \infty} +\infty.
\end{equation*}
\end{proposition}
This is a consequence of the proof of \cite[Theorem 2]{SY99}. See also \cite[Proof of Lemma 5: Case 2]{BN11}. We give a more geometric interpretation this fact. 
\begin{proof}[Proof of \cref{pr:deg:blowup}]
From \cref{pr:deg:powersharp} we find for each \(k \in \Nset\), a map $f_k \in C^\infty(\S^n,\S^n)$ such that
\begin{equation*}
 \deg f_k = k
\end{equation*}
and
\begin{equation*}
 [f_k]_{W^{s,\frac{n}{s}}(\S^n)} \aleq k^{\frac{s}{n}}.
\end{equation*}
Set $g_k := k^{-\sigma} f_k$, then
\begin{equation*}
 [g_k]_{W^{s,\frac{n}{s}}(\S^n)} \aleq k^{\frac{s}{n}-\sigma}.
\end{equation*}
Setting $\omega=x^1 dx^2\wedge\ldots\wedge dx^{n+1}$ the volume form of $\S^n$ (extended to a function $\omega \in C_c^\infty(\Ep^n T^\ast \R^{n+1})$) we then have
\begin{equation*}
 \int_{\S^n} g_k^\ast\omega = k^{-\sigma(n+1)}\, k \xrightarrow{k \to \infty} \infty \quad \text{if $\sigma < \frac{1}{n+1}$}.
\end{equation*}
Taking $\sigma = \frac{s}{n}$ we thus get the desired sequence whenever $s < \frac{n}{n+1}$. 
\end{proof}
From \cref{pr:deg:blowup} one could think that the condition $s > \frac{n}{n+1}$ in the estimate \eqref{eq:degest} was sharp, but this turns out to be false: the following was shown in \cite[Theorem 0.6]{BBM05}.
\begin{theorem}\label{th:degBBM}
Let $f \in C^\infty(\S^n,\S^n)$ then \eqref{eq:degest} holds for any $s \in (0,1]$.
\end{theorem}
\cref{pr:deg:blowup} and \cref{th:degBBM} do not contradict each other: the main point in \cref{pr:deg:blowup} is that it is not assumed that $f_k$ maps into $\S^n$. Indeed from the construction one sees that the maps $f_k$ eventually collapse to zero as $k \to \infty$. Let us summarize these results for the degree as follows
\begin{enumerate}
 \item Without any restriction on the topology, for every map $f \in C^\infty(\S^n,\R^{n+1})$
 \begin{equation}\label{eq:fastomegaest}
  \left | \int_{\S^n} f^\ast\omega \right | \aleq [f]_{W^{s,\frac{n}{s}}(\S^n)}^{\frac{n}{s}}
 \end{equation}
holds when $s \geq \frac{n}{n+1}$. This estimate may fail for $s < \frac{n}{n+1}$.
\item With the additional topological restriction $f: \S^n \to \S^n$ \eqref{eq:fastomegaest} holds for any $s > 0$.
\end{enumerate}
In this sense the differentiability $\frac{n}{n+1}$ is the sharp limit case from the Harmonic Analysis point-of-view, while (2) is the situation from the topological point of view.

\section{Hopf Degree estimates}\label{s:nasty}
Our proof of \cref{th:Hdeg} is based on Harmonic Analysis, namely commutator estimates. Coifman--Lions--Meyer--Semmes showed that Jacobians (and more generally div-curl terms) are related to commutator estimates (in particular the Coifman--Rochberg--Weiss commutator \cite{CRW}) and obtained Hardy spaces estimates for Jacobians. Similar effects had also been observed in terms of Wente's inequality \cite{Wente69,BC84,T85}. Extending these arguments, fractional Sobolev space estimates for Jacobians have been obtained by Sickel and Youssfi,  \cite{SY99}. These fractional Sobolev space estimates for Jacobian can be proven by an elegant argument using trace space characterizations and harmonic extension, \cite{Brezis-Nguyen-2011}, and indeed also the Hardy-space estimates and more generally the Coifman--Rochberg--Weiss estimates can be obtained by an extension argument \cite{LS18}. 

Since all these arguments are written in Euclidean Space, we will use the stereographic projection to pull back our definition of the Hopf degree to $\R^{4n-1}$.
\begin{lemma}
  \label{lemma_Stereographic}
  Let \(\omega \in  C^\infty(\Ep^{2n} T^\ast \R^{2n+1})\), \(f \in C^\infty (\S^{4n - 1}, \R^{2n})\).
  and \( \eta \in C^\infty (\Ep^{2n-1}T^\ast \S^{4n - 1})\).
  If \(d\eta = f^* \omega\) on \(\S^{4n - 1}\) and if \( \Upsilon : \R^{4n - 1} \to \Sset^{4n - 1}\)
  denotes the inverse stereographic projection on the equatorial plane defined for each point
  \(x \in \R^{4n - 1}\) by 
  \begin{equation*}
    \Upsilon (x) 
    \defeq
    \Bigl(\frac{2 x}{1 + \abs{x}^2}, \frac{1 - \abs{x}^2}{1 + \abs{x}^2} \Bigr),
  \end{equation*}
  then \(d \Upsilon^* \eta = d ((f \compose \Upsilon)^* \omega)\),
  \begin{equation*}
  \int_{\R^{4n - 1}}\Upsilon^* \eta
  \wedge
    (f \compose \Upsilon)^* \omega 
    =
    \int_{\Sset^{4 n - 1}} \eta \wedge f^* \omega 
  \end{equation*}
and for any $s \in (0,1)$,
\begin{equation*}
\begin{split}
[f\circ\Upsilon]_{W^{s,\frac{4n-1}{s}}(\R^{4n-1})} = [f]_{W^{s,\frac{4n-1}{s}}(\S^{4n-1})}
\end{split}\end{equation*}
\end{lemma}

\begin{proof}
We first observe that by classical properties of the pullback of differential forms, we have 
\(d \Upsilon^* \eta = 
\Upsilon^* d \eta
= \Upsilon^* d (f^* \omega)
= d (\Upsilon^* f^* \omega)
= d ((f \compose \Upsilon)^* \omega)
\)
and 
\begin{equation*}
\begin{split}
  \int_{\R^{4n - 1}}\Upsilon^* \eta
  \wedge
(f \compose \Upsilon)^* \omega 
&=
\int_{\R^{4n - 1}}
(\Upsilon^* \eta)
\wedge 
(\Upsilon^* f^* \omega) \\
&=
\int_{\R^{4n - 1}}
\Upsilon^* (\eta \wedge f^*\omega)
= 
\int_{\Sset^{4n - 1}}
\eta \wedge f^* \omega .
\end{split}
\end{equation*}
  We note that for every \(x, y \in \R^{4 n - 1}\),
  \begin{equation*}
  \abs{\Upsilon (y) - \Upsilon (x) }
  = \frac{2 \abs{y - x}}{\sqrt{(1 + \abs{x}^2)(1 + \abs{y}^2)}}
  \end{equation*}
  and for each \(x \in \R^{4n - 1}\), $v \in T\R^{4n-1}$
\begin{equation*}
\abs{\langle D \Upsilon (x),v\rangle} = \frac{2 \abs{v}}{1 + \abs{x}^2},
\end{equation*}
and thus 
\begin{equation*}
\abs{\det D \Upsilon (x)} = \frac{4}{(1 + \abs{x}^2)^{4 n - 1}}.
\end{equation*}
Hence we have, by the change of variable formula
\begin{equation*}
\begin{split}
  \iint\limits_{\R^{4n - 1} \times \R^{4 n - 1}}\hspace{-1em}&
\frac{\abs{f (\Upsilon(y)) - f (\Upsilon(x))}^{\frac{4n-1}{s}}}{\abs{y - x}^{8n - 2}}\dif y \dif x\\
&=
\iint\limits_{\R^{4n - 1} \times \R^{2 n - 1}}\hspace{-1em}
\frac{\abs{f (\Upsilon(y)) - f (\Upsilon(x))}^{\frac{4n-1}{s}}}{\abs{\Upsilon (y) - \Upsilon (x)}^{8n - 2}}
\abs{\det D \Upsilon (x)} \abs{\det D \Upsilon (y)} \dif y \dif x\\
&=
\iint\limits_{\Sset^{4n - 1} \times \Sset^{4 n - 1}}
\frac{\abs{f (y) - f (x)}^{\frac{4n-1}{s}}}{\abs{y - x}^{8n - 2}}\dif y \dif x.\qedhere
\end{split}
\end{equation*}
\end{proof}
%
We denote by $\lapms{s}$ the Riesz potential on \(\R^{m}\),
that we let act on a $k$-form $\alpha \in C^\infty(\Ep^k T^\ast \R^m)$ component-wise.
That is if we write 
\begin{equation*}
 \alpha = \sum_{1 \leq i_1 < \ldots < i_k \leq m} \alpha_{i_1,\ldots, i_k}\, dx^{i_1} \wedge \ldots \wedge dx^{i_k},
\end{equation*}
with $\alpha_{i_1,\ldots, i_k} \in C^\infty(\Omega)$, then 
\begin{equation*}
 \lapms{s} \alpha := \sum_{1 \leq i_1 < \ldots < i_k \leq m} \lapms{s}\alpha_{i_1,\ldots, i_k}\, dx^{i_1} \wedge \ldots \wedge dx^{i_k}.
\end{equation*}
With this notation we have the following estimate, which is the crucial estimate underlying our argument.
%
\begin{proposition}\label{pr:Jaces}
If \(f \in C^\infty_c (\R^m, \R^\ell)\) and 
\(\kappa \in C_c^\infty (\bigwedge^k T^\ast \R^\ell)\), then for every \(s \in (\frac{1}{2}, 1)\),
\begin{multline*}
\int_{\R^m} \abs{\lapms{1/2} f^\ast \kappa}^2 
\\
\aleq 
[f]_{W^{1 - \frac{1}{2 k},2k}(\R^{m})}^{2k}
\norm{\kappa}_{L^\infty (\Rset^\ell)}^{2 - \frac{1}{s}}
\bigl(\norm{\kappa}_{L^\infty (\Rset^\ell)}+ \norm{D \kappa}_{L^\infty (\Rset^\ell)} [f]_{W^{s, m/s}(\R^{m})}\bigr)^{\frac{1}{s}}.
\end{multline*}
\end{proposition}
\begin{remark}\label{rem:prjaces}
The estimate is sharp in the following sense. 
In the proof of \cite[Theorem 2, necessity of (16)]{SY99} for any $t < 1-\frac{1}{2k}$ they construct a sequence $f_i \in C^\infty(\R^{n},\R^{k})$ (actually in the Schwartz class) such that 
\begin{equation*}
\sup_{i} \|f_i\|_{L^\infty} + [f_i]_{W^{t,\frac{n}{t}}} < +\infty,
\end{equation*}
but if $J_k f_i$ denotes the determinant of the first $k\times k$ submatrix of $Df_i$ then
\begin{equation*}
 \|\lapms{1/2} (J_k f_i )\|_{H^{-\frac{1}{2}}(\R^n)} \xrightarrow{i \to \infty} +\infty.
\end{equation*}
If we set $F_i := (f_i^1,\ldots,f^k_i,1)$ and $\beta(x) := x^{k+1}$, $\ell = k+1$, this leads to
\begin{equation*}
 \|\lapms{1/2} \brac{\beta(F) (J_k F_i)}\|_{H^{-\frac{1}{2}}(\R^n)} \xrightarrow{i \to \infty} \infty.
\end{equation*}
\end{remark}

\begin{proof}[Proof of \cref{pr:Jaces}]
By linearity, we can consider the case where 
\begin{equation*}
 \kappa = \tilde{h} \, \theta_1 \wedge \dotsb \wedge \theta_k,
\end{equation*}
for \(\tilde{h} \in C^\infty (\R^\ell)\) and \(\theta_1, \dotsc, \theta_k \in \Ep^1 T^*\Rset^m\).
Hence 
\begin{equation*}
   f^\ast \kappa
   = \tilde{h}\circ f \ f^*(\theta_1 \wedge \dotsb \wedge \theta_k).
\end{equation*}
For simplicity of notation we set $h := \tilde{h}\circ f$.

Let \(\varphi \in C^\infty_c (\bigwedge^{m - k} \R^m)\) with $\|\varphi\|_{L^2(\Ep^{m-k} \R^m)} \leq 1$. We want to estimate 
\begin{equation}
  \label{eq_iem3poigh3ief7O}
\int_{\R^m} f^\ast \kappa \wedge \lapms{1/2} \varphi
=
\int_{\R^m} h \ f^*\theta_1 \wedge \dotsb \wedge f^* \theta_k \wedge \psi,
\end{equation}
where we have set \(\psi := \lapms{1/2} \varphi\).

We define the functions \(F: \R^m \times (0, +\infty) \to \R^k\), \(H: \R^m \to \Rset\) and \(\Psi: \R^m \to \bigwedge^{m - k} \R^m\) be the harmonic extensions of the functions \(f\), \(h\) and \(\psi\) to
\(\R^{m + 1}_+ \defeq \R^m \times (0, +\infty)\) defined by the Poisson kernel.
For instance, we have for each \((x, t) \in \R^m \times (0, +\infty)\),
\begin{equation}
 \Psi (x, t) = 
 c_m
 \int_{\R^{m}} 
 \frac{t \, \psi (y)}{(t^2 + \abs{x - y}^2)^\frac{m + 1}{2}} \dif y
 .
\end{equation}

By the definition of \(H\) and \(\Psi\) through the Poisson kernel, we have the decay at infinity,
\begin{equation*}
\abs{H (x, t)} \abs{\Psi (x, t)} \leq \frac{C(h,\psi)}{(\abs{x}^2 + t^2 + 1)^m},
\end{equation*}
and thus 
\begin{equation*}
(F^* \theta_1 \wedge  \dotsb \wedge F^* \theta_k \,H \Psi) (x, t)
\le C(f,h,\psi) \frac{1}{(\abs{x}^2 + t^2 + 1)^m}.
\end{equation*}
From \eqref{eq_iem3poigh3ief7O} Cartan formula or Stokes theorem yield
\begin{equation*}
\int_{\R^m} f^\ast \kappa \wedge \lapms{1/2} \varphi
=(-1)^{k} 
\int_{\R^{m + 1}_+} F^* \theta_1\wedge  \dotsb \wedge F^* \theta_k \wedge (dH \wedge \Psi + H d\Psi).
\end{equation*}
We have thus by the Cauchy--Schwarz and the triangle inequality
\begin{multline}
  \label{eq_Haucael3le4aFoe}
\int_{\R^m} \int_{\R^m} f^\ast \kappa \wedge \lapms{1/2} \varphi\\
\aleq 
\biggl(\int_{\R^{m + 1}_+} \abs{D f}^{2 k} \Biggr)^\frac{1}{2}
\Biggl(\biggl(\int_{\R^{m + 1}_+} \abs{DH}^2 \abs{\Psi}^2\biggr)^\frac{1}{2}
+ \biggl(\int_{\R^{m + 1}_+} \abs{H}^2 \abs{D \Psi}^2\biggr)^\frac{1}{2} \Biggr).
\end{multline}
We now estimate successively the three integrals on the right-hand side of \eqref{eq_Haucael3le4aFoe}.
We first have, by properties of the harmonic extension that 
\begin{equation}
  \label{eq_IeMievoot5kooce}
\int_{\R^{m + 1}_+} \abs{D f}^{2 k}
\aleq [f]_{W^{1 - \frac{1}{2 k},2k}(\R^{m})}^{2k}.
\end{equation}

Next, we have 
\begin{equation}
  \label{eq_Eeyediqu1Ied8xo}
\int_{\R^{m + 1}_+} \abs{H}^2 \abs{D \Psi}^2
\aleq \norm{H}_{L^\infty}^2 [\psi]_{H^{1/2} (\R^m)}^2
\aeq \norm{h}_{L^\infty}^2 \norm{\varphi}_{L^2 (\R^m)}^2.
\end{equation}

Finally, we observe that for every \((x, t) \in \R^{m + 1}_+\), we have 
\begin{equation*}
\begin{split}
\abs{\Psi (x,t)} 
&\aleq \int_{\R^m} \frac{t}{(t + \abs{x - y})^{m + 1}} \abs{\psi (y)} \dif y
\aleq \int_{\R^m} \int_{\abs{x - y}}^{+\infty} \frac{t}{(t + r)^{m + 2}} \dif y\\
&\aleq \int_0^{+\infty} \frac{t r^m}{(t + r)^{m + 2}} \fint_{B_r (x)} \abs{\psi (y)} \dif y
\aleq  \int_0^{+\infty} \frac{t}{(t + r)^{2}} \dif r \mathcal{M} \psi (x)
\aleq \mathcal{M} \psi (x),
\end{split}
\end{equation*}
and thus 
\begin{equation}
  \label{eq_aG1ierihahquieb}
  \begin{split}
  \int_{\R^{m + 1}_+} \abs{DH}^2 \abs{\Psi}^2
  &\aleq  \int_{\R^m} \abs{\mathcal{M} \psi (x)}^2
   \int_{0}^{+\infty}\abs{D H (x, t)}^2 \dif t\dif x \\
   &\aleq \Biggl(\int_{\R^m} \abs{\mathcal{M} \psi (x)}^\frac{2 m}{m - 1} \Biggr)^{1 - \frac{1}{m}}
   \Biggl(\int_{\R^m} \biggl(\int_{0}^{+\infty}\abs{D H (x, t)}^{2} \dif t\biggr)^m \Biggr)^\frac{1}{m}.
   \end{split}
\end{equation}
By the maximal function theorem, \cite[Theorem 2.1.6.]{GrafakosC}, and Sobolev-type inequalities for Riesz potentials, \cite[Section~6.1.1]{GrafakosM}, or see \cite{Stein}, we have 
\begin{equation}
  \label{eq_jaikui9ra8Uz6eo}
\Biggl(\int_{\R^m} \abs{\mathcal{M} \psi (x)}^\frac{2 m}{m - 1} \Biggr)^{\frac{1}{2} - \frac{1}{2m}}
\aleq \norm{\psi}_{L^\frac{2m}{m - 1}}
= \norm{\lapms{1/2} \varphi}_{L^\frac{2m}{m - 1} (\R^m)}
\aleq \norm{\varphi}_{L^2 (\R^m)}.
\end{equation}
Moreover, in view of the characterization of the homogeneous Triebel--Lizorkin spaces $\dot{F}^{s}_{p,q}(\R^m)$ by harmonic extensions, cf. \cite[Theorem~10.8.]{LS18},
\begin{equation}
  \label{eq_aish2lau9aetahZ}
  \brac{\int_{\R^{m}} \brac{ \int_0^{+\infty} |D H(x,t)|^2 dt}^{m}}^{\frac{1}{2 m}} 
\aeq [h]_{F^{\frac{1}{2}}_{2m ,2}(\R^{m})}.
\end{equation}
Combining \eqref{eq_aG1ierihahquieb}, \eqref{eq_jaikui9ra8Uz6eo} and \eqref{eq_aish2lau9aetahZ}, we obtain
\begin{equation}
  \label{eq_Uz0daicoabu7sie}
  \begin{split}
    \int_{\R^{m + 1}_+} \abs{DH}^2 \abs{\Psi}^2
    &\aleq \norm{\varphi}_{L^2 (\R^m)} [h]_{F^{\frac{1}{2}}_{2m ,2}(\R^{m})}.
  \end{split}
\end{equation}
By inserting the inequalities \eqref{eq_IeMievoot5kooce}, \eqref{eq_Eeyediqu1Ied8xo} and \eqref{eq_Uz0daicoabu7sie} into \eqref{eq_Haucael3le4aFoe}, we have proved now that 
\begin{equation*}
  \biggabs{\int_{\R^m} f^\ast \kappa \wedge \lapms{1/2} \varphi}
  \aleq
  \norm{\varphi}_{L^2 (\R^m)}[f]_{W^{1 - \frac{1}{2 k},2k}(\R^{m})}^{k}
  \bigl(\norm{h}_{L^\infty} + \norm{h}_{F^{1/2}_{2m, 2} (\R^m)} \bigr).
\end{equation*}
and hence 
\begin{equation}
  \label{eq_ux7vaeNeichaih3}
  \int_{\R^m} \abs{ \lapms{1/2} f^\ast \kappa }^2
  \aleq [f]_{W^{1 - \frac{1}{2 k},2k}(\R^{m})}^{2k}
  \bigl(\norm{h}_{L^\infty}^2 + \norm{h}_{F^{1/2}_{2m, 2} (\R^m)}^2 \bigr).
\end{equation}
Now we use Gagliardo--Nirenberg inequalities for Triebel spaces, \cite[Proposition~5.6]{BM17}. For any $s \in (\frac{1}{2}, 1)$, we have 
\begin{equation}
  \label{eq_eeN6PhaighaeQuu}
\begin{split}
[h]_{F^{1/2}_{2 m ,2}(\R^{m})} 
&\aleq  
[h]_{F^{s}_{m/s ,m/s}(\R^{m})}^{1/2} 
[h]_{F^{1 - s}_{m/(1 - s), m/(1-s)}(\R^{m})}^{1/2}\\
&\aeq  
[h]_{W^{s,m/s}(\R^{m})}^{1/2}
[h]_{W^{1 - s, m/(1 - s)}(\R^{m})}^{1/2}.
\end{split}
\end{equation}
Moreover, by the Gagliardo--Nirenberg inequality for fractional Sobolev spaces, \cite{BM17}, we have since \(s \ge \frac{1}{2}\)
\begin{equation}
  \label{eq_ae8sa5och4sei8A}
[h]_{W^{1 - s, m/(1 - s)}(\R^{m})}
\aleq [h]_{W^{s, m/s}(\R^{m})}^{(1 -s)/s}
\norm{h}_{L^\infty}^{(2s - 1)/s},
\end{equation}
and thus 
\begin{equation}
  \label{eq_Phae5og6aehaing}
  [h]_{F^{1/2}_{2 m ,2}(\R^{m})} 
  \aleq  [h]_{W^{s, m/s}(\R^{m})}^{\frac{1}{2s}}
  \norm{h}_{L^\infty}^{1 - \frac{1}{2 s}}.
\end{equation}
Since furthermore $h = \tilde{h} \circ f$, and $\tilde{h}$ is $C_c^\infty$, we have shown
\begin{equation}
  \label{eq_Phae5og6aehaing2}
  [h]_{F^{1/2}_{2 m ,2}(\R^{m})} 
  \aleq  C \norm{D \Tilde{h}}_{L^\infty (\Rset^\ell)}^\frac{1}{2s} \, \norm{\Tilde{h}}_{L^\infty (\Rset^\ell)}^{1- \frac{1}{2s}} [f]_{W^{s, m/s}(\R^{m})}^{\frac{1}{2s}}.
\end{equation}
The conclusion follows from \eqref{eq_ux7vaeNeichaih3} and \eqref{eq_Phae5og6aehaing2}.
\end{proof}

\begin{proof}[Proof of \cref{th:Hdeg}]
We assume in view of \cref{lemma_Stereographic}
 that 
 $f \in C_c^\infty(\R^{4n-1},\R^{2n+1})$, $\omega \in C_c^\infty(\Ep^{2n} T^\ast \R^{2n+1})$ 
 and $f^\ast (d\omega) = 0$.
Recall that we denote by $\lapms{2}$ the Newton potential (or Riesz potential of order \(2\)) on \(\R^{4n - 1}\). Set 
\begin{equation}
  \label{eq_nohyefaidae3uYeed}
  \theta := \lapms{2} f^* \omega,
\end{equation}
where \(\lapms{2}\) acts on each component of \(f^* \omega\).
Then \(\theta \in C^\infty(\Ep^{2n} T^\ast \R^{4n-1})\) and
\begin{equation*}
\lap \theta = f^\ast\omega.
\end{equation*}
Here $\lap= dd^* + d^* d$ is the Laplace-Beltrami operator on differential forms.
Observe that by assumption $df^\ast\omega =dd\eta= 0$. Thus,
\begin{equation}
  \label{eq:laplaceforms}
\lap d\theta = (dd^\ast + d^\ast d)d \theta = 
d d^* d \theta 
= 
d(d d^\ast + d^\ast d) \theta 
=
d \Delta \theta 
= 
d f^\ast (\omega)
= 0.
\end{equation}
From the definition of $\theta$ in \eqref{eq_nohyefaidae3uYeed} and since \(f^*\omega\) is compactly supported we have for every \(x \in \Rset^{4n - 1}\),
\begin{equation*}
\abs{d\theta (x)}
\aleq
\frac{1}{1 + \abs{x}^{4 n - 2}}.
\end{equation*}
Since the Laplace-Beltrami in Euclidean space $\R^{4n-1}$ acts as the usual Laplacian on the coefficients of the form, see \cite[\textsection 6.35, Exercise 6, p.\thinspace{}252]{W71}, and in view of \eqref{eq:laplaceforms}, we can apply we can apply Liouville's theorem for harmonic functions
and obtain that $d\theta \equiv 0$ on \(\R^{4n - 1}\). Hence,
\begin{equation*}
f^\ast\omega = \Delta \theta = (d d^* + d^* d)\theta = d d^\ast \theta \quad \text{on \(\R^{4n - 1}\)}.
\end{equation*}
Since $f^\ast\omega = d\eta=dd^\ast \theta$, 
\begin{equation*}
    \int_{\R^{4n - 1}} 
        (\eta -d^*\theta) 
      \wedge 
        f^* \omega  
  =
  \int_{\R^{4n - 1}} 
      (\eta -d^*\theta) 
    \wedge 
      d\eta  
  = 
    \int_{\R^{4n - 1}} 
      d(\eta -d^*\theta)\wedge \eta 
  =
  0.
\end{equation*}
It follows from \eqref{eq_nohyefaidae3uYeed},
\begin{equation*}
\begin{split}
  \int_{\R^{4n - 1}} 
      \eta 
    \wedge 
      f^* \omega
  =&
  \int_{\R^{4n-1}} 
      d^\ast \theta 
    \wedge 
      dd^\ast \theta \\
=& 
  \int_{\R^{4n-1}} 
      d^\ast (\lapms{2} f^* \omega)  
    \wedge
      d d^\ast (\lapms{2} f^* \omega)
= 
  \int_{\R^{4n-1}} 
      \lapms{2} (d^\ast f^* \omega)
    \wedge
      \lapms{2} (d d^\ast f^* \omega)\\
=&
  \int_{\R^{4n - 1}}
      \lapms{3/2} (d^\ast f^* \omega)
    \wedge
      \lapms{5/2} (d d^\ast f^* \omega)\\
=&   
  \int_{\R^{4n - 1}}
      d^\ast \lapms{3/2} (f^* \omega)
    \wedge
      d d^\ast \lapms{5/2} (f^* \omega).      
\end{split}
\end{equation*}
Thus,
\begin{equation*}
    \abs{
      \int_{\R^{4n - 1}} 
        \eta
      \wedge 
        f^* \omega 
    }
  \aleq 
    \|d^\ast \lapms{3/2} (f^* \omega)\|_{L^2(\R^{4n-1})}
    \, 
    \|d d^\ast \lapms{5/2} (f^* \omega)\|_{L^2(\R^{4n-1})}\, .
\end{equation*}
Now we observe that we can express $d^\ast \lapms{3/2} = T_1 \lapms{1/2}$ and $dd^\ast \lapms{5/2} = T_2 \lapms{1/2}$ where $T_1$ and $T_2$ are Calderon-Zygmund operators (essentially they are a collection of Riesz transforms). From the boundedness of these operators on $L^2(\R^{4n-1})$ we obtain
\begin{equation*}
    \abs{
      \int_{\R^{4n - 1}} 
        \eta
        \wedge 
        f^* \omega 
    } 
    \leq \|\lapms{1/2} (f^\ast \omega)\|_{L^2(\R^{4n-1})}^2.
\end{equation*}
We apply \cref{pr:Jaces} with \(m = 4n - 1\) and \(\ell = 2n\) and obtain, for any $s \geq \frac{1}{2}$,
\begin{multline}
\biggl\lvert 
  \int_{\R^{4n - 1}} 
  \eta
  \wedge 
  f^* \omega 
\biggr\rvert 
\aleq [f]_{W^{1 - \frac{1}{4 n},4 n}(\R^{4 n - 1})}^{4 n}
\norm{\omega}_{L^\infty (\Rset^{2n})}^{2 - \frac{1}{s}} \\
\bigl(\norm{\omega }_{L^\infty (\Rset^{2n})} + \norm{D\omega}_{L^\infty (\Rset^{2n})}[f]_{W^{s, (4n - 1)/s}(\R^{4n - 1})}\bigr)^{\frac{1}{s}}.
\end{multline}
%
This establishes \cref{th:Hdeg} for $s = \frac{4n-1}{4n}$. For $s \in (\frac{4n-1}{4n},1]$ we use the Gagliardo-Nirenberg estimate,
\begin{equation*}
[f]_{W^{1- \frac{1}{4n},4n}(\R^{4n-1})}^{4n} \aleq \|f\|_{L^\infty(\R^{4n-1})}^{4n-\frac{4n-1}{s}} [f]_{W^{s,\frac{4n-1}{s}}(\R^{4n-1})}^{\frac{4n-1}{s}}.
\end{equation*}
We obtain then 
\begin{multline}
  \biggl\lvert 
  \int_{\R^{4n - 1}} 
  \eta
  \wedge 
  f^* \omega 
  \biggr\rvert 
  \le \|f\|_{L^\infty(\R^{4n-1})}^{4n-\frac{4n-1}{s}} 
  \norm{\omega}_{L^\infty (\Rset^{2n})}^{2 - \frac{1}{s}} \cdot \\
  \cdot \brac{\norm{\omega }_{L^\infty}^\frac{1}{s} [f]_{W^{s,\frac{4n-1}{s}}(\R^{4n-1})}^{\frac{4n-1}{s}} + \norm{D\omega}_{L^\infty (\Rset^{2n})}^\frac{1}{s} [f]_{W^{s, \frac{4n - 1}{s}}(\R^{4n - 1})}^{4n}}.\qedhere
\end{multline}
\end{proof}

\section{Sharpness of the Hopf Degree estimates}
The power $\frac{4n}{s}$ in the estimate of \cref{th:Hdeg} is sharp in the following sense
\begin{proposition}\label{pr:degh:powersharp}
For \(n \in \N\), $s \in (0,1]$ and any $d \in \Z$ there exists a map $f_d: \S^{4n-1} \to \S^{2n}$ satisfying
\begin{equation*}
[f_d]_{W^{s,\frac{4n}{s}}}^{\frac{4n - 1}{s}} \leq C(n,s)\, |\deg_H f_d|.
\end{equation*}
\end{proposition}

The proof of \cref{pr:degh:powersharp} follows closely the strategy of Rivi\`ere for \(n = 1\) and \(s = 1\) \cite[Lemma III.1]{R98}. We extend it to dimension $n \geq 1$. Then the fractional Gagliardo--Nirenberg interpolation inequality, for example see \cite{BM17}, implies the estimate of \cref{pr:degh:powersharp}. 

The key construction is provided by application of the Whitehead integral formula to the Whitehead product of a map from \(\Sset^{2n}\) to \(\Sset^{2n}\) with itself.

\begin{lemma}
  \label{proposition_deg_hopf_2}
For every \(n \in \Nset\) and \(k \in \Nset\), there exists a map \(f \in C^\infty(\Sset^{4n - 1}, \Sset^{2n})\) such that 
\begin{align*}
   \deg_H (f) &= 2 k^2&
   &\text{ and }&
   \norm{Df}_{L^\infty} \aleq k^{\frac{1}{2n}}.
\end{align*}
\end{lemma}
\begin{proof}[Proof of \cref{proposition_deg_hopf_2}]%
Let \(g \in C^\infty (\R^{2n}, \Sset^{2n})\), be a map such that 
\(g = a_*\) in \(\R^{2n} \setminus \Bset^{2n}\) for some point \(a_* \in \Sset^{2n}\).
Since \(\R^{2n} / (\R^{2n} \setminus \Bset^{2n})\) is homeomorphic to 
\(\Sset^{2n}\backslash \{a_*\}\), the map \(g\) has a well-defined Brouwer degree that can be computed as 
\begin{equation*}
\deg g = \int_{\R^{2n}} g^* \omega_{\Sset^{2n}},
\end{equation*}
where \(\omega_{\Sset^{2n}} \in C_c^\infty (\bigwedge^{2n} T^\ast \R^{2n+1})\) is a volume form on \(\Sset^{2n}\) such that \(\int_{\Sset^{2n}} \omega_{\Sset^{2n}} = 1\).
By \cref{lem:deg:powersharp}, we can choose \(g\) in such a way that 
\(
 \deg g = k
\)
and
\(\norm{Dg}_{L^\infty} \aleq k^{-1/2n}\).

We remark that
\begin{equation*}
   \S^{4n-1} 
 = 
   \bigl\{
     (x_+,x_-) \in \R^{2n} \times \R^{2n}
     :\;  |x_+|^2+|x_-|^2 = 1
   \bigr\}.
\end{equation*}
We define the sets 
\begin{align*}
 \S_\pm
 &= \bigl\{(x_+,x_-) \in \S^{4n - 1} :\; \abs{x_\pm} > \abs{x_\mp}\bigr\},
\end{align*}
and we observe that 
\begin{equation}
  \label{eq_thi5we5voo5eGheel}
    \partial \S_+ = \partial \S_- 
  = 
    \bigl\{(x_+,x_-) \in \S^{4n - 1} :\; \abs{x_+} = \abs{x_-}\bigr\} 
  = 
    \Sset^{2n - 1}_{1/\sqrt{2}} \times \Sset^{2n - 1}_{1/\sqrt{2}} .
\end{equation}
We define maps \(P_\pm : \Sset^{4n - 1} \to \Rset^{2n}\) by
\(
P_\pm (x_+, x_-) = \sqrt{2} \,x_\pm
\).
Set
\begin{equation*}
f (x)
\defeq
  \begin{cases}
    g (P_+ (x)) & \text{if \(x \in \Sset_+\)},\\
    g (P_- (x)) & \text{if \(x \in \Sset_-\)},\\
    a_* & \text{if \(x \in  \Sset^{2n - 1}_{1/\sqrt{2}} \times \Sset^{2n - 1}_{1/\sqrt{2}}\)}. 
  \end{cases}
\end{equation*}
We observe that $f$ is continuous since for $(x,y) \in  \Sset^{2n - 1}_{1/\sqrt{2}} \times \Sset^{2n - 1}_{1/\sqrt{2}}$, 
\(g (P_+ (x)) =g (P_- (x)) = a_*\). 
We have immediately that \(\norm{Df}_{L^\infty (\Sset^{4n - 1})} \aleq k^{-\frac{1}{2n}}\).

Since \(g^* \omega_{\Sset^{2n}} = 0\) is a $2n$-form in $\R^{2n}$, it is closed, and by Poincar\'e Lemma we can find \(\eta \in C^\infty (\Ep^{2n - 1} T^\ast \Rset^{2n})\) such that 
\(d \eta = g^* \omega_{\Sset^{2n}}\).
We observe then that 
\begin{equation*}
  f^\ast\omega_{\Sset^{2n}} 
  = 
  P_+^\ast g^\ast\omega +P_-^\ast g^\ast\omega
  = d (P_+^\ast \eta + P_-^\ast \eta).
\end{equation*}

The Hopf degree of $f$ can now be computed as
\begin{equation*}
  \deg_H (f) = \int_{\S^{4n-1}} (P_+^\ast \eta + P_-^\ast \eta)\wedge  d (P_+^\ast \eta + P_-^\ast \eta).
\end{equation*}
We observe that 
\(
  P_\pm^\ast \eta \wedge d P_\pm^\ast \eta
  = P_\pm^\ast (\eta \wedge d \eta) = 0,
\) since $\eta \wedge d\eta$ is a $4n-1$ form of $T^\ast \R^{2n}$.
Thus 
\begin{equation*}
  \deg_H (f) = \int_{\S^{4n-1}} P_+^\ast \eta \wedge dP_-^\ast \eta 
  +
  \int_{\S^{4n-1}} P_-^\ast \eta \wedge dP_+^\ast \eta.
\end{equation*}
Since 
\(
  \supp dP_\pm^\ast \eta \subset \S_{\pm}
\),
we have by the Stokes--Cartan formula, in view of \eqref{eq_thi5we5voo5eGheel},
\begin{equation*}
  \int_{\S^{4n-1}} P_\pm^\ast \eta \wedge dP_\mp^\ast \eta = \int_{\S_\mp} P_\pm^\ast \eta \wedge dP_\mp^\ast \eta
  = \int_{\S^{2n-1}_{1/\sqrt{2}}\times \S^{2n-1}_{1/\sqrt{2}}} P_\pm^\ast \eta
  \wedge P_{\mp}^* \eta
  = \biggl(\int_{\S^{2n - 1}} \eta\biggr)^2.
\end{equation*}
Now again by Stokes
\begin{equation*}
  \int_{\S^{2n - 1}} \eta
  = \int_{\Bset^{2n}} d \eta = \int_{\Rset^{2n}} g^\ast \omega_{\Sset^{2n}}
  = \deg g = k
\end{equation*}
from which we conclude.
\end{proof}
Now we follow the strategy in \cite[Lemma III.1]{R98} to obtain
\begin{proof}[Proof of \cref{pr:degh:powersharp}]
Given \(d \in \Nset\) we choose \(k \in \N\) such that 
\((k - 1)^2 \le \abs{d} <  k^2\)
and we let \(f_d\) be given by \cref{proposition_deg_hopf_2} for this \(k\).
Hence we have,
\begin{align}
\abs{\deg_H (f_d)} &\ge \abs{d}&
&\text{ and }&
  \label{eq_zid6raeZereej6I}
  \abs{D f_d} &\aleq k^\frac{1}{2n} \aleq d^{\frac{1}{4n}}.
\end{align}
We now estimate by \eqref{eq_zid6raeZereej6I}
\begin{equation}
  \label{eq_Aas4eesohs3waig}
  \int_{\Sset^{4n - 1}}
  \abs{D f_d}^{4n - 1}
  \aleq 
  d^{1 - \frac{1}{4n}}.
\end{equation}
In view of \eqref{eq_Aas4eesohs3waig} and of the fractional Gagliardo--Nirenberg embedding \cite{BM17}, we have
\begin{equation*}
\begin{split}
[f_d]_{W^{s, \frac{4n - 1}{s}}}^\frac{4n - 1}{s}
\aleq
\norm{f_d}_{L^\infty}^{(4n - 1)(\frac{1}{s} - 1)}
\norm{f_d}_{W^{1, 4n - 1} (\Sset^{4n - 1})}^{4n - 1}
\aleq d^{1 - \frac{1}{4n}}. 
\end{split}
\end{equation*}
Taking the power $\frac{4n}{4n-1}$ on both side of the estimate, we conclude.
\end{proof}
Since the methods from Harmonic Analysis we use in \cref{th:Hdeg} are sharp, one might expect that an estimate as in in \cref{th:Hdeg} is not true for $s < \frac{4n-1}{4n}$ without additional topological restrictions. We are not able to prove this, but can only show that \cref{th:Hdeg} fails for $s< \frac{2n}{2n+1} < \frac{4n-1}{4n}$.
\begin{proposition}\label{pr:degh:blowup}
For any $s < \frac{2n}{2n+1}$ there exists a sequence of maps $f_k \in W^{s,\frac{4n-1}{s}}\cap L^\infty (\S^{4n-1},\R^{2n+1})$ such that 
\begin{equation*}
\sup_{k \in \N} \|f_k\|_{L^\infty} + [f_k]_{W^{s,\frac{4n-1}{s}}(\S^{4n-1})} < +\infty
\end{equation*}
but for $d\eta = f^\ast\omega_{\Sset^{2n}}$, we have
\begin{equation*}
\left |\int_{\S^{4n-1}} f_k^\ast\omega \wedge \eta \right | \xrightarrow{k \to \infty} +\infty.
\end{equation*}
\end{proposition}
\begin{proof}
From \cref{pr:degh:powersharp} we find for a sequence $k \to \infty$ maps $f_k \in C^\infty(\S^{4n-1},\S^{2n})$ such that
\begin{align*}
 [f_k]_{W^{s,\frac{4n-1}{s}}(\S^{2n})} &\aleq k^{\frac{s}{4n}}&
&\text{ and }&
 \deg_H f_k &\geq k.
\end{align*}
That is, for each $\eta_k \in C^\infty(\Ep^{2n-1}T^\ast \S^{4n-1})$ such that
\(
d\eta_k = f^\ast\omega_{\Sset^{2n}}\), we have 
\begin{equation*}
 \int_{\S^{4n-1}} \alpha_k \wedge f^\ast\omega = k.
\end{equation*}
Set $g_k := k^{-\sigma} f_k$, then
\begin{equation*}
 [g_k]_{W^{s,\frac{4n-1}{s}}(\S^{4n-1})} \aleq k^{\frac{s}{4n}-\sigma}.
\end{equation*}
On the other hand, if \(d\theta_k = g_k^\ast\omega_{\Sset^{2n}}\), then 
\(d (k^{\sigma(2n+1)}\theta_k) = f_k^* \omega_{\Sset^{2n}}\) and thus 
\begin{equation*}
 \int_{\S^{4n-1}} \theta_k \wedge g_k^\ast\omega = k^{-\sigma 2(2n+1)+1}.
\end{equation*}
We conclude that if we pick $\sigma = \frac{s}{4n}$ we have
\begin{equation*}
\sup_{k \in \Nset} [g_k]_{W^{s,\frac{4n-1}{s}}(\S^{4n-1})} < +\infty,
\end{equation*}
and if $s <\frac{2n}{2n+1} = 1-\frac{1}{2n+1}$,
\begin{equation*}
 \int_{\S^{4n-1}} \beta_k \wedge g_k^\ast\omega = k^{-\frac{s}{4n} 2(2n+1)+1} \xrightarrow{k \to \infty} +\infty.
 \qedhere
\end{equation*}
\end{proof}

\bibliographystyle{amsabbrv}%
\bibliography{bib}%

\end{document}